\documentclass[a4paper,12pt]{article}
\usepackage[colorlinks=true, linkcolor=blue, citecolor=blue]{hyperref}
\usepackage[margin=1.2in]{geometry}

\usepackage{helvet}
\usepackage[latin1]{inputenc}

\usepackage{amsmath,txfonts,fourier}
\usepackage{amssymb}
\usepackage{mathtools}
\usepackage{enumerate}
\usepackage{graphicx}
\usepackage{float}
\usepackage{cite}
\usepackage{color}
\usepackage{siunitx}
\usepackage{booktabs}
\usepackage{tabularx}
\usepackage{changepage}
\usepackage{orcidlink}
\usepackage{url}
\usepackage{caption}
\usepackage{subcaption}
\usepackage{array}
\usepackage{rotating}
\usepackage{graphicx}
\usepackage{epstopdf}
\usepackage{graphicx}
\usepackage{epstopdf}
\usepackage{algorithmic}
\usepackage{relsize}
\usepackage{tikz}
\usepackage{comment}

\usepackage{authblk}
\usepackage{enumitem}
\usepackage{tikz}
\usepackage{tkz-euclide}
\setlist[itemize]{leftmargin=.5in}

\usepackage{amssymb}
\usepackage{hyperref}
\usepackage{algorithm}
\usepackage{geometry}
\usepackage{esint}
\usepackage{array}
\usepackage[linewidth=2pt]{mdframed}

\newmdenv[
topline=false,
bottomline=false,
rightline=false,
skipabove=\topsep,
skipbelow=\topsep,
linewidth=4
]{siderules}
\usepackage{hyperref}
\usepackage{xcolor}
\usepackage[most]{tcolorbox}
\usepackage[nottoc]{tocbibind}

\newtheorem{theorem}{Theorem}[section]

\newtheorem{lemma}[theorem]{Lemma}

\newenvironment{proof}{\paragraph{Proof:}}{\hfill$\square$}

\newcommand{\mat}[1]{\begin{pmatrix}#1\end{pmatrix}}
\newcommand{\eps}{\varepsilon}
\newcommand{\w}{\mathbf{w}}
\newcommand{\z}{\mathbf{z}}
\newcommand{\bra}[1]{\left(#1\right)}

\title{Turing Instability Suppressed and Induced by Multiplicative Noise in Brusselator System}

\author[1]{Qasim Khan}
\author[2]{Anthony Suen}
\author[3]{Bao Quoc Tang\footnote{Corresponding author.}}
\affil[1]{\small Department of Mathematics and Information Technology,
	The Education University of Hong Kong,
	10 Lo Ping Road, Tai Po, N.T,
	Hong Kong.\break
	\href{mailto:qasimkhan@s.eduhk.hk}{qasimkhan@s.eduhk.hk}}
\affil[2]{\small Department of Mathematics and Information Technology,
	The Education University of Hong Kong,
	10 Lo Ping Road, Tai Po, N.T,
	Hong Kong.\break
	\href{mailto:acksuen@eduhk.hk}{acksuen@eduhk.hk}}
\affil[3]{\small Department of Mathematics and Scientific Computing, University of Graz, Graz, Austria\break  \href{mailto:quoc.tang@uni-graz.at}{quoc.tang@uni-graz.at}}

\date{}

\begin{document}
	\maketitle
	\begin{abstract}
		The effect of multiplicative noise to the Turing instability of the Brusselator system is investigated. We show that when the noise acts on both of the concentrations with the same intensities, then the Turing instability is suppressed provided that the intensities are sufficiently large. This aligns with the stabilizing effect of multiplicative noise in partial differential equations. Utilizing the linearized system, we can quantify the magnitude of noise which stabilizes the system. On the other hand, when the noise is involving only one concentration, then the Turing instability can be triggered with suitable intensities. These are confirmed by numerical simulations.
	\end{abstract}
	\begin{keywords}
		Brusselator system; Turing instability; Multiplicative noise; Instability induced/suppressed by noise. \\
	\end{keywords}
	%

	\section{Introduction}
	Reaction-diffusion system (RDS) provides a compelling and useful modeling tool for many problems in real life. One of the main reasons is that, even a simple RDS can exhibit many complex patterns as shown in the well known work \cite{pearson1993complex}. Revealing the mechanism for the formation of these patterns is the objective of many works since then or even earlier. One prime example is the famous work of Turing \cite{turing1952chemical} where he argued that one possible mechanism is the complicated interplay between diffusion and (chemical) reactions in terms of activators and inhibitors. More precisely, it was shown that if the  activator has short range diffusion and inhibitor has long range diffusion, then the stable spatially homogeneous equilibrium becomes unstable, consequently leading to spatially inhomogeneous patterns. This has initiated a now classical topic named \textit{Turing instability}, and is the subject of a vast number of research works (see refer the interested reader to the book \cite{murray2007mathematical} for more details). 
	
	\medskip
	The Turing instability has been confirmed in certain real life biological systems, see for instance \cite{miyazawa2010blending} where experiments on zebrafish stripes showed that they are formed as an interaction of two types of cells, playing roles of activator and inhibitor, following Turing mechanism. However, the different scales of diffusion might not be ``physical'' in biological systems cf. \cite{haas2021turing}, so researchers looked for additional features that suppress or trigger Turing instability, including systems in time evolving domains \cite{madzvamuse2010stability,van2021turing}, systems with spatially inhomogeneous coefficients \cite{van2021pattern,kozak2019pattern}, or the effect of random noise. For instance, at the level of reaction-diffusion master equations, \cite{hori2012noise} showed that noise can induce spatial patterns for chemical reactions, and \cite{dobramysl2018stochastic} studied the effects of noise on patterns for RDS or population dynamics. Some existing works looked at the effect of additive noise, for instance, \cite{kolinichenko2024noise} showed induced pattern for thermochemical kinetics, \cite{biancalani2017giant} investigated giant amplification of noise in fluctuation-induced pattern formation, and \cite{bashkirtseva2021stochastic} considered the sensitivity of Turing patterns. For multiplicative noise, \cite{zheng2017turing} used some explicit stochastic noises and Taylor expansion to obtain their effect on Turing bifurcation, while \cite{xiao2023effects} investigated the effects of noise on critical points of Turing instability on networks. Stochastic Turing patterns triggered by noise have been even confirmed through experiments in \cite{karig2018stochastic}. 
	
	\medskip
	For scalar PDE, multiplicative noise has been known to have stabilizing effect, see e.g. \cite{kwiecinska1999stabilization,caraballo2001stabilization,caraballo2006stabilization,caraballo2007effect}. Especially, in \cite{kwiecinska1999stabilization}, by looking at the Fourier analysis of a linear PDE, it was shown that the stabilization is because the multiplicative noise shifts the spectral of the Laplacian to the left with the magnitude proportional to the (square of) noise intensities. In a way, this could be seen as scaling of the range of diffusion. For systems, it is therefore natural to ask: {\it what happens if the scaling for different species are the same or different?} In this paper, we investigate this question by looking at the PDE Brusselator system with multiplicative noise. It will be demonstrated that, if the noise intensities are the same for different species, multiplicative noise has suppression effect on Turing instability. On the other hand, if the noise is acting on only one species, it could induce instability even if the short range activator - long range inhibitor is not fulfilled for deterministic systems. The suppression effects are proved by analyzing a linearized systems and using some Fourier analysis, while the induced instability is shown via numerical simulations. It is remarked that the analysis can be carried out for other RDS concerning Turing instability. We would like to also point out several works studying the effect of noise in Brusselator system in different contexts, see e.g. \cite{yerrapragada1997analysis,bashkirtseva2000sensitivity,engel2024noise} and references therein.

	\medskip
	The manuscript is structured as follows: Section \ref{s1} presents the stochastic Brusselator with multiplicative noise. The linear analysis is given in Section \ref{sec3}, while  Section \ref{s2} is devoted to numerical simulations for the nonlinear system. In the last Section we give some discussion and outlook.
	\section{Stochastic Brusselator system}\label{s1}
	
	The Brusselator is characterized by the reactions
	\begin{align*}
		& \mathsf A \to \mathsf  X\\
		& 2\mathsf  X + \mathsf  Y \to 3 \mathsf {X} \\
		& \mathsf  B + \mathsf  X \to \mathsf  Y + \mathsf  D\\
		& \mathsf  X \to \mathsf  E
	\end{align*}
	for certain chemicals $\mathsf  A, \mathsf  B, \mathsf  X, \mathsf  Y, \mathsf  D$ and $\mathsf  E$. It is assumed that $\mathsf  A, \mathsf  B$ are in large excess, and therefore their concentrations can be taken as constants, say $A>0$ and $B>0$. We arrive at the following reaction-diffusion system for the concentrations of $X$ and $Y$, which are denoted by $u$ and $v$, respectively,
	\begin{equation}\label{B-system}
		\begin{cases}
			\partial_t u = d_u\Delta u + A - (B+1)u + u^2v, &x\in \Omega, t> 0,\\
			\partial_t v = d_v\Delta v + B u - u^2v, &x\in\Omega, t>0,\\
			\nabla u\cdot \nu(x) = \nabla v \cdot \nu(x) = 0,  &x\in\partial\Omega, t>0,\\
			u(x,0) =u_0(x), \; v(x,0) = v_0(x), &x\in\Omega,
		\end{cases}
	\end{equation}
	where $d_u, d_v> 0$ are positive diffusion coefficients. Here $\Omega\subset \mathbb R^n$, $n\ge 1$, is an open, bounded domain with smooth boundary $\partial\Omega$ (in case $n\ge 2$), and $\nu(x)$ is the outward normal vector at $x\in\partial\Omega$.
	Due to the homogeneous Neumann boundary conditions, it is straightforward to see that $(u_\infty, v_\infty) = (A, B/A)$ is the unique spatially homogeneous steady state of \eqref{B-system}. Moreover, if $B < 1 + A^2$, this steady state is asymptotically stable. It is also well-known, see e.g. \cite{perthame2015parabolic} that in this case, Turing instability occurs when $d_u \ll d_v$, leading to the formation of patterns, while the steady state $(u_\infty, v_\infty)$ is still asymptotically stable when $d_u \approx d_v$. In this work, we are interested in the effect of multiplicative noise concerning this Turing instability. More precisely, we consider the following
	stochastic Brusselator system with multiplicative noise
	\begin{equation}\label{B_system}
		\begin{cases}
			d u = (d_u\Delta u + A - (B+1)u + u^2v)dt + \sigma_u(u- u_\infty) dW_t, &x\in \Omega, t> 0,\\
			d v = (d_v\Delta v + B u - u^2v)dt + \sigma_v (v-v_\infty) dW_t, &x\in\Omega, t>0,\\
			\nabla u\cdot \nu = \nabla v \cdot \nu = 0, &x\in\partial\Omega, t>0,\\
			u(x,0) =u_0(x), \; v(x,0) = v_0(x), &x\in\Omega,
		\end{cases}
	\end{equation}
	where $W_t$ is a real-valued Wiener process, and $\sigma_u, \sigma_v \in \mathbb R$ are noise intensities. The forms of the noise are so that $(u_\infty, v_\infty)$ is also a steady state of the stochastic system. 
	
	
	\section{Linear stability analysis for the stochastic system} \label{sec3}
	
	The linearized system of \eqref{B_system} around the equilibrium $(u_\infty, v_\infty)$ reads as
	\begin{equation}\label{LR_sys}
		\begin{cases}
			du = (d_u\Delta u + (B-1)u + A^2v)dt + \sigma_u u dW_t, &x\in\Omega, t>0,\\
			dv = (d_v\Delta v - Bu - A^2v)dt + \sigma_v v dW_t, &x\in\Omega, t>0,\\
			\nabla u\cdot \nu = \nabla v \cdot \nu = 0, &x\in\partial\Omega, t>0,\\
			u(x,0) =u_0(x) - u_\infty, \; v(x,0) = v_0(x) - \infty, &x\in\Omega.
		\end{cases}
	\end{equation}
	
	Denoting $a = B - 1, b = A^2, c = -B$ and $d = -A^2$, we rewrite the linearized system in the general form 
	\begin{equation}\label{s2}
		\begin{cases}
			du = (d_{u} \Delta u + au+bv)dt+\sigma_u u d W_t,\\
			dv = (d_v \Delta v + cu+dv)dt+\sigma_v v d W_t.
		\end{cases}
	\end{equation}
	For simplicity, we consider the case where $u$ is the activator and $v$ is the inhibitor, meaning that $a>0$ and $d<0$, which consequently imposes $B>1$. In this case, it is well-known that, if $d_u/d_v$ is sufficiently small, then Turing instability occurs (see e.g. \cite[Theorem 7.1]{perthame2015parabolic}).
	
	\medskip
	First we have a look at the corresponding deterministic differential system reads as
	\begin{equation}\label{ODE}
		\mat{\bar u\\ \bar v}' = \mat{a & b\\ c&d}\mat{\bar u\\ \bar v} =: M\mat{\bar u\\ \bar v}.
	\end{equation}
	Due to the condition $B < 1+A^2$, it follows that
	\begin{equation*}
		\mathcal T:= \text{Trace}(M) = B - 1 - A^2 < 0, \quad \mathcal D:= \text{Det}(M) = A^2 > 0.
	\end{equation*}
	Therefore, the zero solution to the ODE system \eqref{ODE} is exponentially stable, i.e. there exist constants $C, \lambda_{\text{ODE}} > 0$ such that
	\begin{equation*}
		|\bar u(t)| + |\bar v(t)| \le Ce^{-\lambda_{\text{ODE}}\cdot t} \quad \forall t\ge 0.
	\end{equation*}
	
	To analyze the stochastic system, we consider the eigenvalue problem with Neumann boundary condition
	\begin{equation*}
		\begin{cases}
			-\Delta \varphi = \mu \varphi, &x\in\Omega,\\
			\nabla \varphi \cdot \nu = 0, &x\in\partial\Omega.
		\end{cases}
	\end{equation*}
	It is well-known that there is a system of eigenvalue-eigenfunction $(\mu_k,\varphi_k)_{k\ge 0}$ such that
	\begin{itemize}
		\item $0 = \mu_1 < \mu_2 \le \mu_3 \le \longrightarrow +\infty$,
		\item $\{\varphi_k\}_{k\ge 1}$ is an orthonormal basis of $L^2(\Omega)$.
	\end{itemize}
	We use the same basic decomposition technique for $u(t)$ and $v(t)$, i.e.,
	\begin{equation}\label{Fourier_expansion}
		u(x,t)=\sum_{k=1}^{\infty}g_{k}(t)\varphi_k(x), \ \  \ v(x,t)=\sum_{k=1}^{\infty}h_{k}(t)\varphi_k(x).
	\end{equation}
	Inserting these into system (\ref{s2}) then multiply by $\varphi_k$ and integrate over $\Omega$, we get
	\begin{eqnarray}
		d \begin{bmatrix}
				g_{k}(t) \\
				h_{k}(t)
		\end{bmatrix}
		=
		\begin{bmatrix}
				a -d_{u} \mu_k &b \\
				c &d-d_{v} \mu_k\\
		\end{bmatrix}
		\begin{bmatrix}
				g_{k}(t) \\
				h_{k}(t)\\
		\end{bmatrix} dt
		+
		\begin{bmatrix}
				\sigma_u &0 \\
				0 &\sigma_v\\
		\end{bmatrix}
			\begin{bmatrix}
				g_{k}(t) \\
				h_{k}(t)\\
		\end{bmatrix} d W_t.
	\end{eqnarray}
	By defining
	\begin{equation*}
		V_k(t):= \begin{bmatrix}
			g_{k}(t) \\
			h_{k}(t)
		\end{bmatrix},
		\quad \mathbf{A}_k:= 
		\begin{bmatrix}
			a -d_{u} \mu_k &b \\
			c &d-d_{v} \mu_k\\
		\end{bmatrix},
		\quad \mathbf{B}:=
		\begin{bmatrix}
			\sigma_u &0 \\
			0 &\sigma_v\\
		\end{bmatrix},
	\end{equation*}
	we arrive to the following stochastic differential system for each $k\ge 1$
	\begin{equation}\label{sto1}
		d{V_{k}(t)}=\mathbf{A}_kV_{k}(t)dt+\mathbf{B}V_{k}(t)d W_t.
	\end{equation}
	
	In the following subsections, we analyze the case of same noise intensities $\sigma_u = \sigma_v$ and different intensities $\sigma_u = 0, \sigma_v \ne  0$ separately.

			\subsection{The case of the same noise intensities}
			
			If $\sigma_u=\sigma_v = \sigma \in \mathbb R$, then $\mathbf{A}_k$ and $\mathbf{B}$ are commutative and we have the formal solution to \eqref{sto1} (see e.g. \cite[Section 3.4]{mao2007stochastic}) as
			\begin{equation}\label{sol1}
				V_k(t)=\exp\left\{\left(\mathbf{A}_k- \frac{\mathbf{B}^2}{2}\right)t \right\}
				\exp(\mathbf{B}  W_t)V_{k,0},
			\end{equation}
			Where
			$\displaystyle 
			V_{k,0} =	\begin{bmatrix}
				u_{k,0} \\
				v_{k,0}\\
			\end{bmatrix}
			$
			with $\displaystyle u_{k,0} = \int_{\Omega} u_0\varphi_k dx$ and $\displaystyle v_{k,0} = \int_{\Omega}v_0 \varphi_k dx$. We now investigate the eigenvalues of $\mathbf{A}_k - \dfrac{1}{2}\mathbf B^2$.
			\begin{lemma}\label{lem1}
				There exists $\sigma_0\in \mathbb R$ and $\omega > 0$ such that for all $\sigma^2 \ge \sigma_0^2$ 
				\begin{equation*}
					\mathfrak{S}\left(\mathbf{A}_k - \dfrac{\mathbf B^2}{2}\right) \subset \{z\in \mathbb C: \, \Re(z) \le -\omega\},\quad \forall k\ge 1,
				\end{equation*}
				where $\mathfrak{S}(M)$ denotes the spectrum of a matrix $M$.
			\end{lemma}
			\begin{proof}
				Denote by $E_{k}=\mathbf{A}_k- \dfrac{\mathbf{B}^2}{2}$. The characteristic polynomial of $E_k$ is
				\begin{equation*}
					\begin{aligned}
						\det(E_{k}-\lambda I) &= \left( a -d_{u} \mu_k-\frac{\sigma^2}{2}-\lambda \right) 
						\left(d-d_{v} \mu_k-\frac{\sigma^2}{2}-\lambda\right)-bc\\
						&=  \lambda^2-\lambda \left[\mathcal T-\left(d_{u}+d_{v}\right)\mu_k -  { \sigma^2 }\right]\\
						&\qquad 
						+\frac{\sigma^4}{4} - \frac{\sigma^2}{2}[\mathcal T - (d_u+d_v)\mu_k] + d_ud_v\mu_k^2 - (ad_v + dd_u)\mu_k + \mathcal D\\
						&=: \lambda^2 - \mathbf w(\mu_k,\sigma)\lambda + \mathbf z(\mu_k,\sigma).
					\end{aligned}
				\end{equation*}
				The eigenvalues $\lambda_{k,1}$ and $\lambda_{k,2}$ of $E_k$ satisfy
				\begin{equation}\label{Viet}
					\lambda_{k,1}+\lambda_{k,2} = \w(\mu_k,\sigma) \quad \text{ and } \quad \lambda_{k,1} \lambda_{k,2} = \z(\mu_k,\sigma).
				\end{equation}
				Since $\mathcal T < 0$ and $\mu_k\ge 0$, we have $\w(\mu_k,\sigma) < 0$ for all $k\ge 1$ and all $\sigma \in \mathbb R$.  We first rewrite $\z(\mu_k,\sigma)$ as
				\begin{equation*}
					\z(\mu_k,\sigma) = \frac{\sigma^4}{4} + \frac{\sigma^2}{2}(-\mathcal{T}) + d_ud_v\mu_k^2 + \bra{\frac{\sigma^2}{2}(d_u+d_v) - (ad_v + dd_u)}\mu_k + \mathcal{D}
				\end{equation*}
				and deduce that, thanks to $\mathcal T< 0$ and $\mathcal D>0$, there exists $\sigma_0 \in \mathbb R$ such that
				\begin{equation*}
					\z(\mu_k,\sigma) > 0 \quad \text{ for all } \quad \sigma^2 \ge \sigma_0^2 \text{ and all } k \ge 1.
				\end{equation*}
				It follows then from \eqref{Viet} that $\Re(\lambda_{k,1}) < 0$ and $\Re(\lambda_{k,2})<0$. If $\mathbf{w}(\mu_k,\sigma)^2 - 4\mathbf{z}(\mu_k,\sigma) \le 0$, then
				\begin{equation*}
					\Re(\lambda_{k,1}) = \Re(\lambda_{k,2}) = \frac{\mathbf{w}(\mu_k,\sigma)}{2} = \frac{\mathcal T - (d_u+d_v)\mu_k - \sigma^2}{2} \le \frac{\mathcal T - \sigma^2}{2} < 0.
				\end{equation*}
				If $\mathbf{w}(\mu_k,\sigma)^2 - 4\mathbf{z}(\mu_k,\sigma) \ge 0$, then $\lambda_{k,1}, \lambda_{k,2}$ are both negative and 
				\begin{equation*}
					\begin{aligned}
						&\max\{\lambda_{k,1}, \lambda_{k,2}\} = \frac 12\bra{\mathbf w(\mu_k,\sigma) + \sqrt{\mathbf{w}(\mu_k,\sigma)^2 - 4\mathbf{z}(\mu_k,\sigma)}\,}\\
						&= \frac 12\bra{ \mathcal T - (d_u+d_v)\mu_k - \sigma^2 + \sqrt{ (d_u-d_v)^2\mu_k^2 + 2(a-d)(d_u - d_v)\mu_k + \mathcal T^2 - 4\mathcal D}\,}\\
						&= \frac 12\bra{ \mathcal T - (d_u+d_v)\mu_k - \sigma^2 + \sqrt{ [(d_u-d_v)\mu_k + B + A^2 - 1]^2 - 4A^2B}\,}\\
						&\le \frac 12\bra{ \mathcal T - (d_u+d_v)\mu_k - \sigma^2 +  (d_u-d_v)\mu_k + B + A^2 - 1}\\
						&=(B-1) - d_v\mu_k -\frac{\sigma^2}{2} \le B - 1 - \frac{\sigma^2}{2}<  0 \quad \forall k\ge 1,
					\end{aligned}
				\end{equation*}
				provided $\sigma^2 > 2(B-1)$. Therefore, we can choose $\sigma_0 \in \mathbb R$ such that $\sigma_0^2 > 2(B-1)$, then for all $\sigma^2 \ge \sigma_0^2$ and for any $k\ge 1$,
				\begin{equation*}
					\mathfrak{S}\left(\mathbf{A}_k - \dfrac{\mathbf B^2}{2}\right) \subset \{z\in \mathbb C: \, \Re(z) \le -\omega\},\quad \forall k\ge 1,
				\end{equation*}
				with
				\begin{equation*}
					\omega:= \min\left\{\frac{-\mathcal T + \sigma^2}{2}, -B + 1 + \frac{\sigma^2}{2} \right\} > 0.
				\end{equation*}
			\end{proof}
			
			\begin{theorem}\label{thm1}
				Assume that $1 < B < 1 + A^2$. Then for any diffusion coefficients $d_u, d_v$, the solution to the linearized Brusselator system \eqref{LR_sys} is exponentially converging to zero as $t\to\infty$ almost surely, provided $2\sigma^2 > B - 1$. Moreover, we have the Lyapunov exponent
				\begin{equation*}
					\limsup_{t\to\infty}\frac{1}{t}\ln(\|u(t)\|_{L^2(\Omega)} + \|v(t)\|_{L^2(\Omega)}) \le -\omega < 0 \text{ almost surely}, 
				\end{equation*}
				with
				\begin{equation*}
					\omega = \min\left\{\frac{-\mathcal T + \sigma^2}{2}, -B + 1 + \frac{\sigma^2}{2} \right\} > 0.
				\end{equation*}
			\end{theorem}
			\begin{proof}
				From \eqref{Fourier_expansion},  Parseval's identity and \eqref{sol1} it follows that
				\begin{equation*}
					\begin{aligned}
						\|u(t)\|_{L^2(\Omega)}^2 + \|v(t)\|_{L^2(\Omega)}^2 &= \sum_{k= 1}^{\infty}|V_k(t)|^2\\
						& = \sum_{k=1}^{\infty}\left|\exp\left\{2\bra{\mathbf A_k - \frac{\mathbf B^2}{2}}t \right\}\right| |\exp(2\mathbf BW_t)| |V_{k,0}|^2.
					\end{aligned}
				\end{equation*}
				Thanks to the property of white noise, we have
				\begin{equation*}
					\lim_{t\to\infty}\frac{\sigma W_t}{t} = 0 \quad \text{a.s.}
				\end{equation*}
				and therefore, for any $\eps>0$ there exists $C_\eps>0$ such that
				\begin{equation*}
					\exp(\sigma W_t) \le C_\eps e^{\eps t} \quad \forall t\ge 0,
				\end{equation*}
				almost surely. Therefore,
				\begin{equation*}
					|\exp(2\mathbf B W_t)| \le C_\eps^2 e^{2\eps t} \quad \forall t\ge 0.
				\end{equation*}
				On the other hand, from Lemma \ref{lem1}, 
				\begin{equation*}
					\left|\exp\left\{2\bra{\mathbf A_k - \frac{\mathbf B^2}{2}}t \right\}\right| \le Ce^{-2\omega t}.
				\end{equation*}
				Thus we can estimate
				\begin{equation*}
					\begin{aligned}
						&\limsup_{t\to\infty}\frac{1}{t}\ln\bra{\|u(t)\|_{L^2(\Omega)} + \|v(t)\|_{L^2(\Omega)}}\\
						&\le  \limsup_{t\to\infty}\frac{1}{2t}\ln\bra{\frac 12\bra{\|u(t)\|_{L^2(\Omega)}^2 + \|v(t)\|_{L^2(\Omega)}^2}}\\
						&\le\limsup_{t\to\infty}\frac{1}{2t}\ln\bra{C_\eps e^{(-2\omega + 2\eps)t} \sum_{k=1}^{\infty}|V_{k,0}|^2} 
						\\
						&\le -\omega + \eps +  \limsup_{t\to\infty}\frac{1}{2t}\ln\bra{C_\eps (\|u_0\|_{L^2(\Omega)}^2 + \|v_0\|_{L^2(\Omega)}^2)}\\
						&= -\omega + \eps 
					\end{aligned}
				\end{equation*}
				almost surely for $\eps>0$ arbitrary. This confirms the Lyapunov exponent and finishes the proof of Theorem \ref{thm1}.
			\end{proof}
			
			\medskip
			Theorem \ref{thm1} shows that with sufficiently large noise intensities, system \eqref{LR_sys} is stable \textit{regardless of diffusion coefficients}. A direct consequence is the suppression of Turing instability by noise. This is demonstrated in Figures \ref{fig:LR_Turing} and \ref{fig:LR_suppression} where we consider the one dimensional case $\Omega = (0,1)$. More precisely, Figure \ref{fig:LR_Turing} shows the evolution of solutions to \eqref{LR_sys} when Turing instability occurs. The parameters in this case 
			\begin{equation*}
				A = 1, \quad B = 1.8, \quad d_u = 5\times 10^{-5}, \quad d_v = 2\times 10^{-3}
			\end{equation*}
			satisfy $B < 1 + A^2$, which makes the ODE \eqref{ODE} stable, but since $d_u/d_v$ is sufficiently small, the PDE system \eqref{LR_sys} is unstable. From initial data being small perturbation of the equilibrium $(0,0)$, the solutions eventually grow to have spatially inhomogeneous profiles, which are shown in Figures \ref{fig:LR_Turing_u} and \ref{fig:LR_Turing_v}. The evolution of the norm in Figure \ref{fig:LR_Turing_norm} shows that, this instability happens after an initial decay phase on the time interval $(0,15)$, and after that the norms of solutions grow exponentially. Moreover, a closer inspection in Figure \ref{fig:LR_Turing_LE_Fourier} shows that there are only a finite number of eigenmodes \textit{in the middle} (between $k = 120$ and $k=200$) which are unstable, while the others stay stable. This can be also shown theoretically (see e.g. \cite{perthame2015parabolic}).
			\begin{figure}[H]
				\centering
				\begin{subfigure}{0.31\textwidth}
					\includegraphics[width=5.2cm]{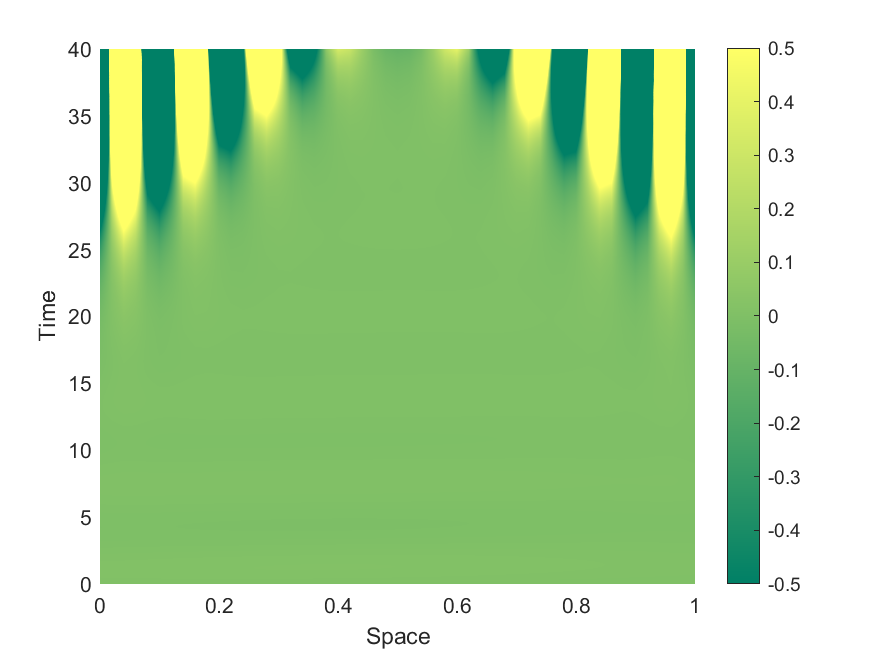}
					\caption{Evolution of $u$}
					\label{fig:LR_Turing_u}
				\end{subfigure} \hspace{.5in}
				\begin{subfigure}{0.31\textwidth}
					\includegraphics[width=5.2cm]{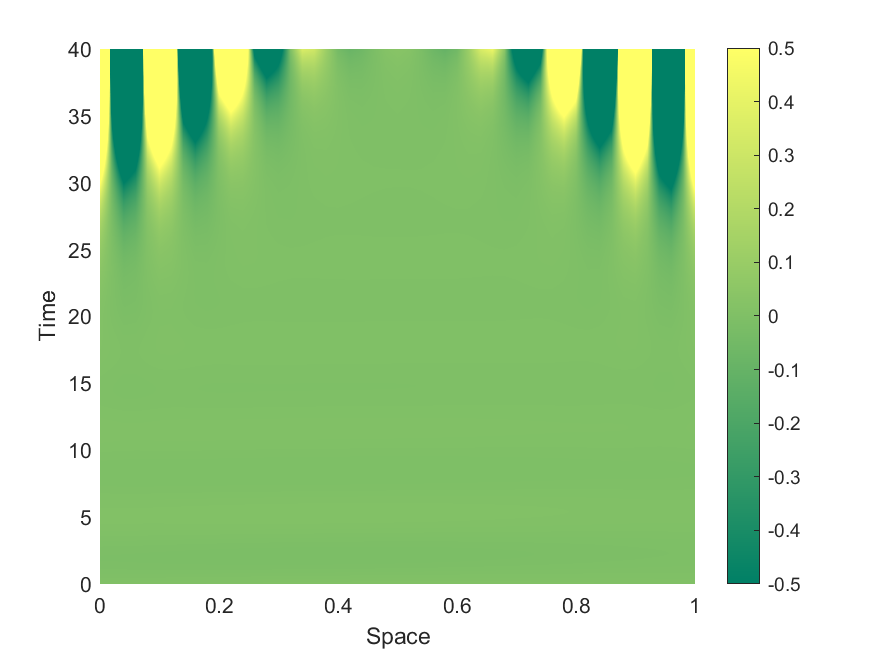}
					\caption{Evolution of $v$}
					\label{fig:LR_Turing_v}
				\end{subfigure}\\
				\begin{subfigure}{0.31\textwidth}
					\includegraphics[width=5.2cm]{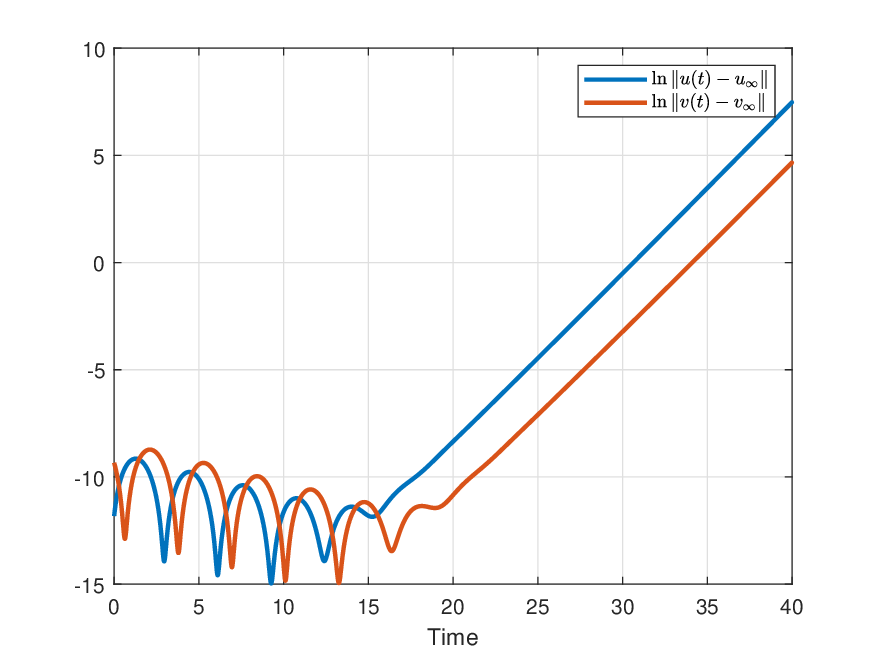}
					\caption{Evolution of the norm}
					\label{fig:LR_Turing_norm}
				\end{subfigure} \hspace{.5in}
				\begin{subfigure}{0.31\textwidth}
					\includegraphics[width=5.2cm]{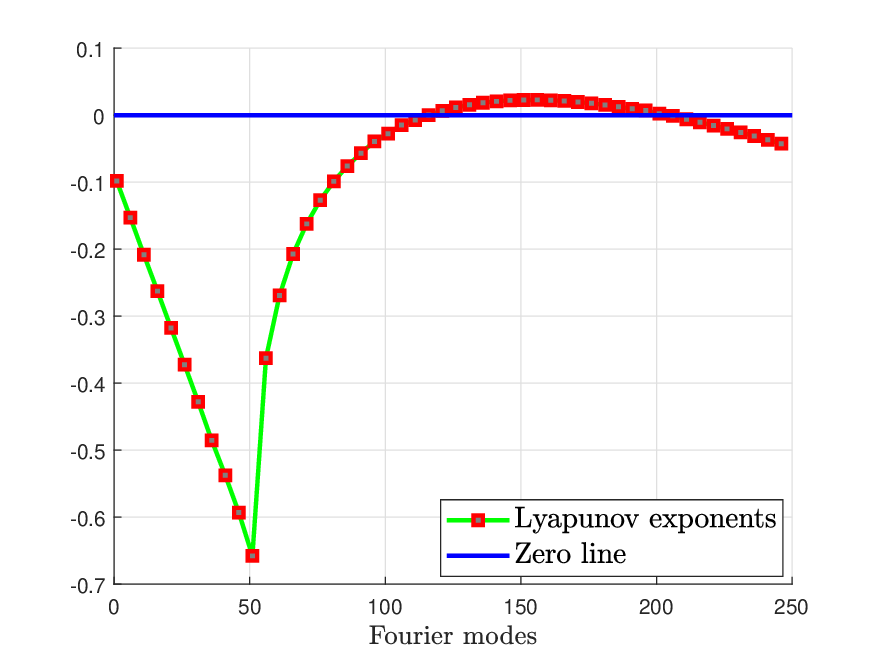}
					\caption{Lyapunov exponents of different Fourier modes}
					\label{fig:LR_Turing_LE_Fourier}
				\end{subfigure}
				\caption{Turing instability without noise of the linearized system \eqref{LR_sys}.}
				\label{fig:LR_Turing}
			\end{figure}
			
			Figure \ref{fig:LR_suppression} shows the suppression of this instability by sufficiently large noise intensities, which are the same for the equations of $u$ and $v$, i.e. $\sigma_u = \sigma_v = \sigma$. In Figure \ref{fig:LR_suppression_norm}, we see $\ln(\|u(t)\|_{L^2(\Omega)} + \|v(t)\|_{L^2(\Omega)})$ behaves almost linearly with negative slopes, which means that the solutions decay to zero exponentially. Moreover, the bigger the noise intensities $|\sigma|$, the faster solutions converge to zero, which is consistent with Theorem \ref{thm1} where the Lyapunov exponent decreases for increasing $\sigma^2$. Figure \ref{fig:LR_suppression_LE} demonstrates the behavior of Lyapunov exponents according to the noise intensity $\sigma$. The green curve with blue points shows the real Lapunov exponent, which becomes negative when $\sigma$ is somewhere in between $0.6$ and $0.8$, and decreases further as $\sigma$ grows, while the blue curve with red points shows an upper bound of the Lyapunov exponent established in Theorem \ref{thm1}. It can be seen that, while being qualitatively accurate, these upper bounds are far from being sharp. Showing optimal bound of Lyapunov exponent for \eqref{LR_sys} depending on $\sigma$ is left for future investigation.
			\begin{figure}[H]
				\centering
				\begin{subfigure}[t]{0.45\textwidth}
					\includegraphics[width=6cm]{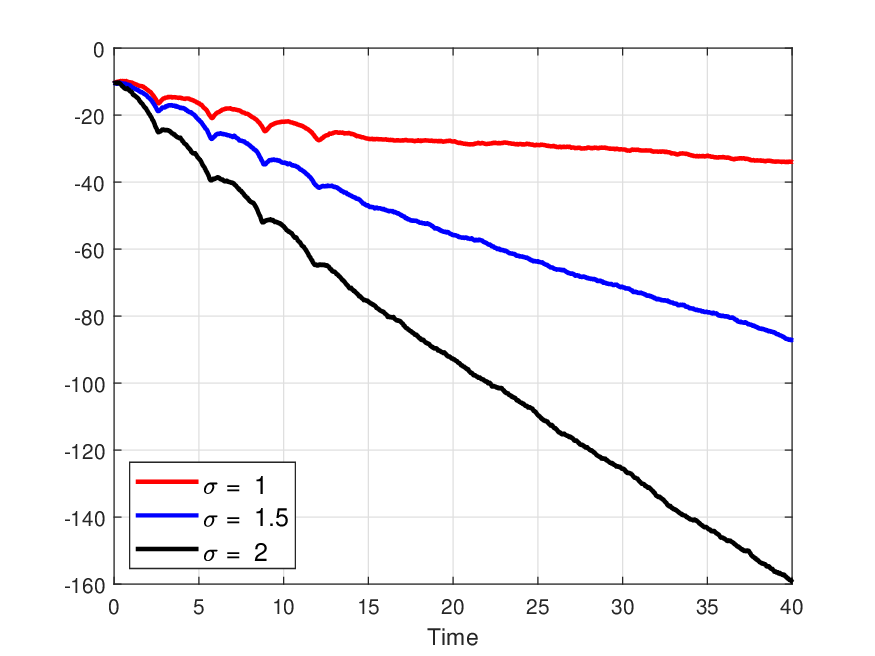}
					\caption{Suppression of Turing instability by noise. The lines present the evolution of $\ln(\|u(t)\|_{L^2(\Omega)} + \|v(t)\|_{L^2(\Omega)})$.}
					\label{fig:LR_suppression_norm}
				\end{subfigure} \hspace*{.1in}
				\begin{subfigure}[t]{0.45\textwidth}
					\includegraphics[width=6cm]{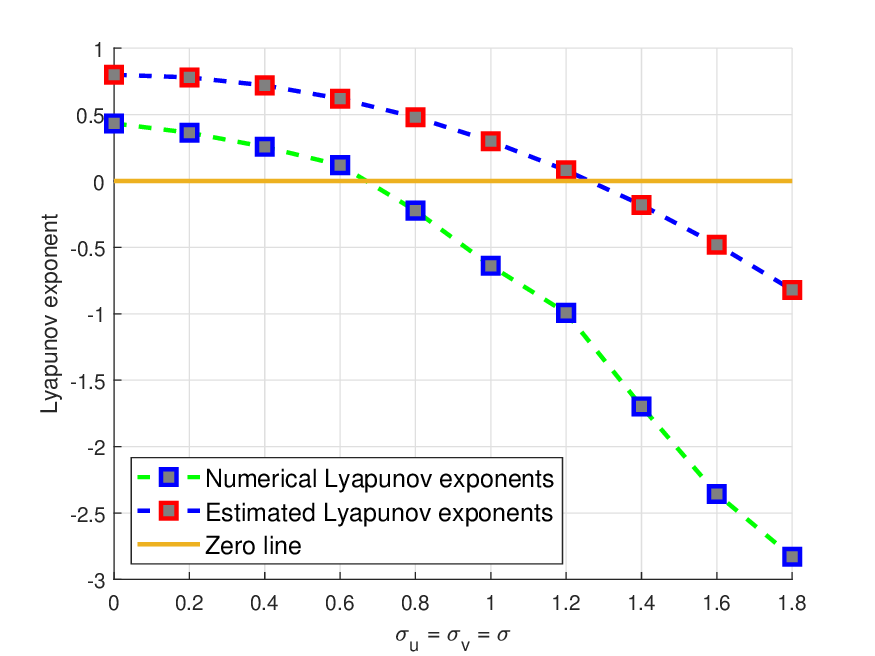}
					\caption{The computed Lyapunov exponents and the explicitly estimated Lyapunov exponents of solutions to \eqref{LR_sys}.}
					\label{fig:LR_suppression_LE}
				\end{subfigure}
				\caption{Suppression by noise with the same noise intensities.}
				\label{fig:LR_suppression}
			\end{figure}
			
			Since \eqref{LR_sys} is linear, the stability, in the presence of noise, is in fact global. This is demonstrated in 
			Figure \ref{fig:LR_global_stability}, where the solutions converge to zero with initial data being far away from the equilibrium $(0,0)$.
			\begin{figure}[H]
				\centering
				\begin{subfigure}{0.31\textwidth}
					\includegraphics[width=5.2cm]{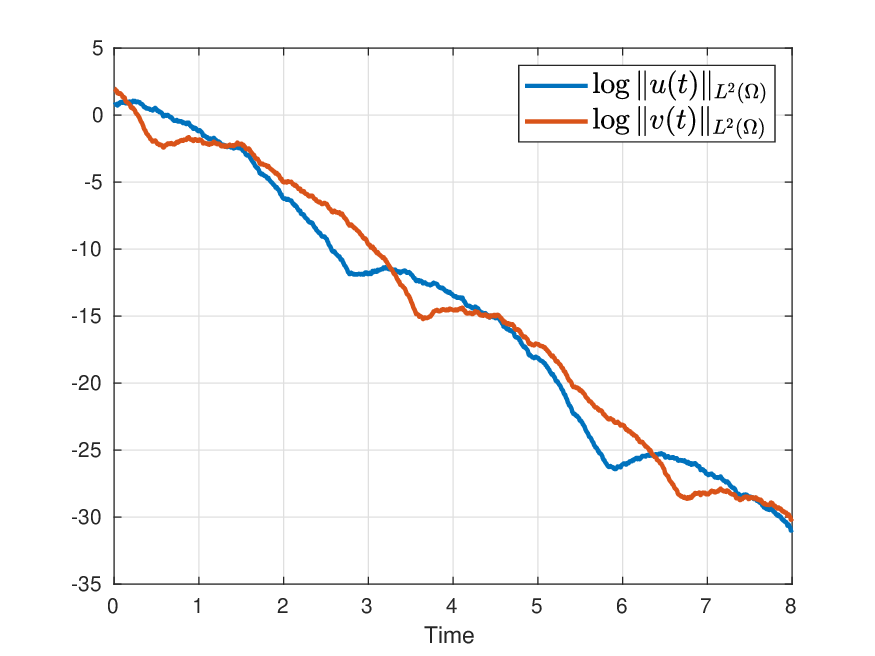}
					\caption{Evolution of the norm}
					\label{fig:LR_global_stability_norm}
				\end{subfigure}
				\begin{subfigure}{0.31\textwidth}
					\includegraphics[width=5.2cm]{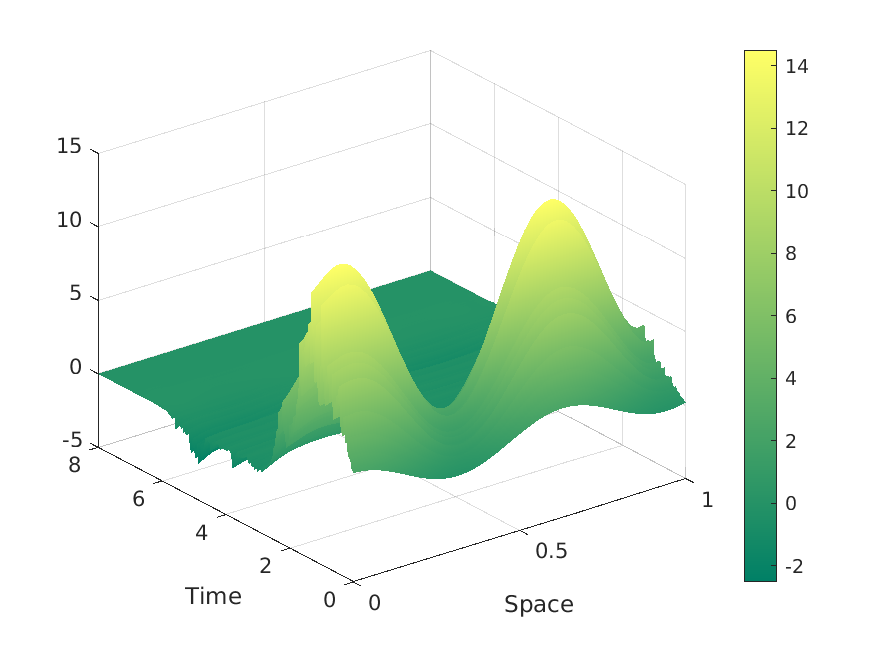}
					\caption{Evolution of $u$}
					\label{fig:LR_global_stability_u}
				\end{subfigure}
				\begin{subfigure}{0.31\textwidth}
					\includegraphics[width=5.2cm]{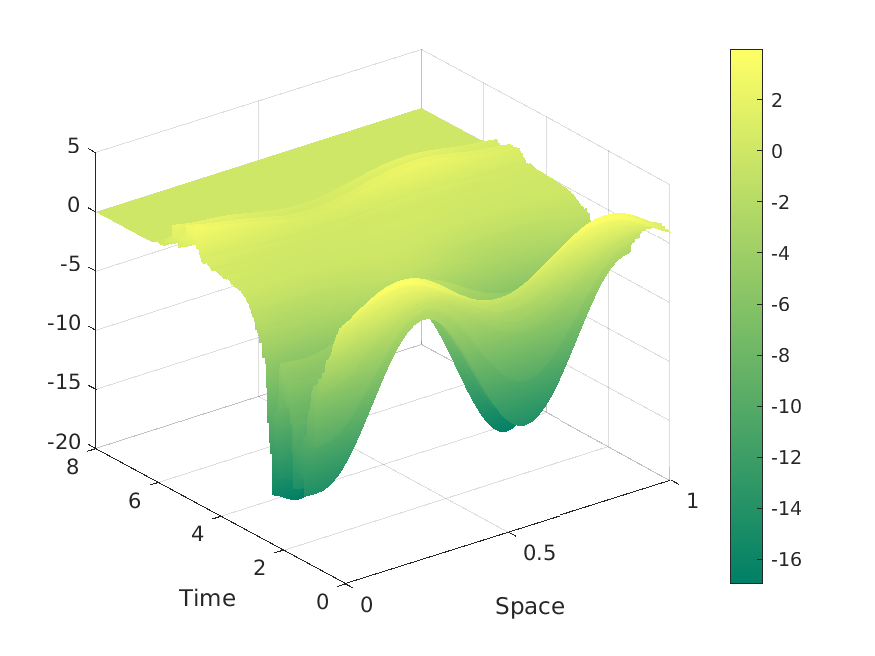}
					\caption{Evolution of $v$}
					\label{fig:LR_global_stability_v}
				\end{subfigure}
				\caption{Global stability of system \eqref{LR_sys} with high noise intensities $\sigma_u = \sigma_v = 2$.}
				\label{fig:LR_global_stability}
			\end{figure}
			
			\subsection{The case of different noise intensities}
			In this part, the noise intensities for $u$ and $v$ are different. For the sake of simplicity, we fix $\sigma_u = 0$ and vary $\sigma_v \ne 0$. Due to this difference, the matrix $\mathbf B$ and $\mathbf A_k$ in \eqref{sto1} are generally no longer commutative, and consequently an explicit expression of solutions similarly to \eqref{sol1} is not possible. Furthermore, due to the degeneracy of $\mathbf B$, the stabilization and destabilization theory in e.g. \cite[Chapter 4]{mao2007stochastic} cannot be applied to determine the stability of \eqref{sto1}. Here, we show by numerical simulations that this difference in noise intensities \textit{destabilizes} the system \eqref{sto1}, and consequently the linear PDE system \eqref{LR_sys}. 
			
			\medskip
			Consider $\Omega = (0,1)$ and the set of parameters
			\begin{equation*}
				A = 1, \quad B = 1.8, \quad d_u = 2\times 10^{-3}, \quad d_v = 1\times 10^{-3},
			\end{equation*}
			which make the ODE system \eqref{ODE} stable since $B < 1 + A^2$. Moreover, since $d_u/d_v$ is large, the PDE system is also stable, i.e. no Turing instability occurs. The destabilization by noise with $\sigma_u = 0$ and $\sigma_v \ne 0$ is shown in Figure \ref{fig:LR_destabilize}, where part (a) shows the evolution of $\ln(\|u(t)\|_{L^2(\Omega)} + \|v(t)\|_{L^2(\Omega)})$ for different values of $\sigma_v$. Here we see that starting from $\sigma_v = 1.5$, the solution starts to grow exponentially. This is confirmed in Figure \ref{fig:LR_destabilize_LE}, where numerical Lyapunov exponents are shown. From this figure, the Lyapunov exponent is monotone increasing in the value of $|\sigma_v|$, and starting from $\sigma_v = 1.5$, the Lyapunov exponent becomes positive.
			
			\begin{figure}[H]
				\centering
				\begin{subfigure}[t]{0.45\textwidth}
					\includegraphics[width=6cm]{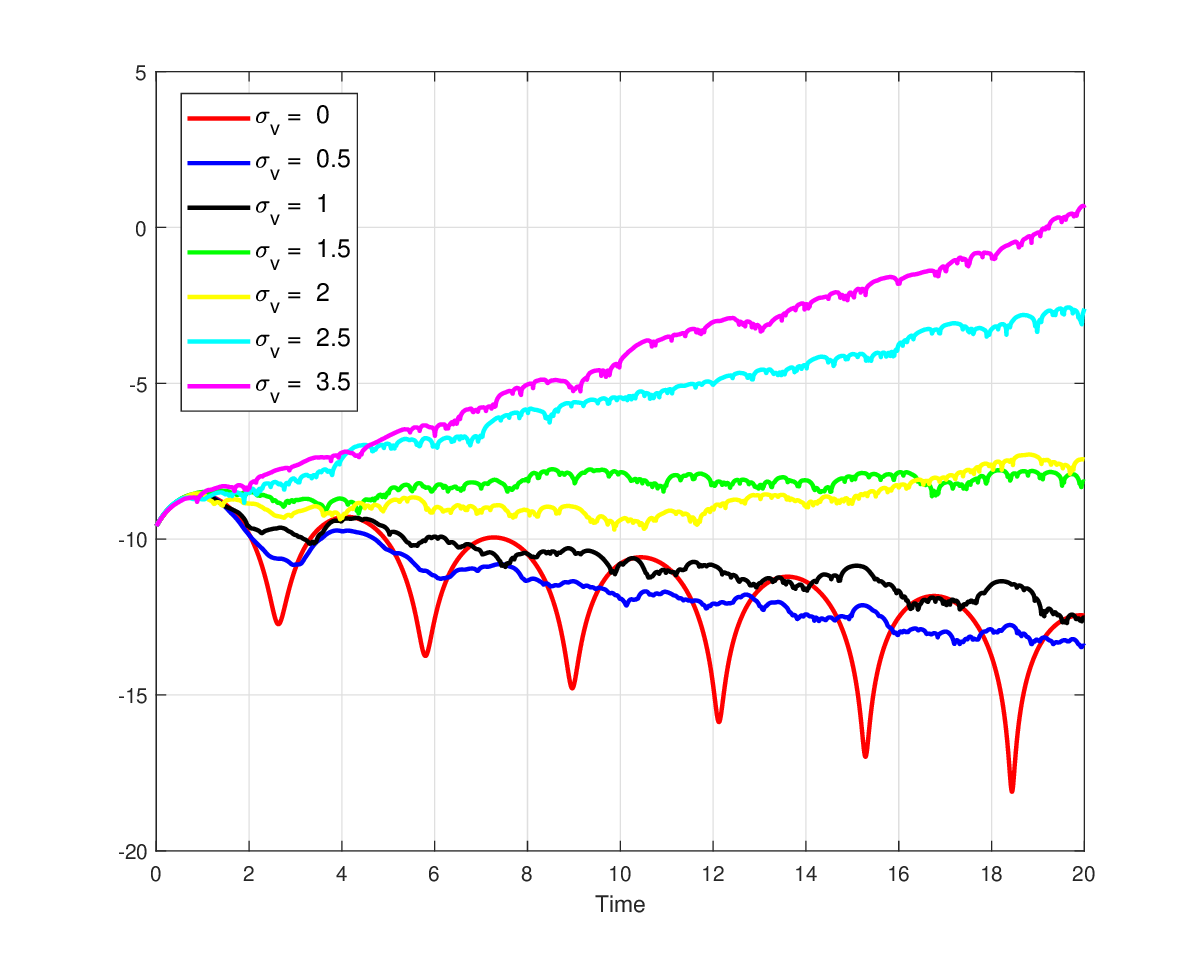}
					\caption{The evolution of $\ln(\|u(t)\|_{L^2(\Omega)} + \|v(t)\|_{L^2(\Omega)})$ according to different values of $\sigma_v$ where $\sigma_u=0$.}
					\label{fig:LR_destabilize_norm}
				\end{subfigure} \hspace*{.1in}
				\begin{subfigure}[t]{0.45\textwidth}
					\includegraphics[width=6cm]{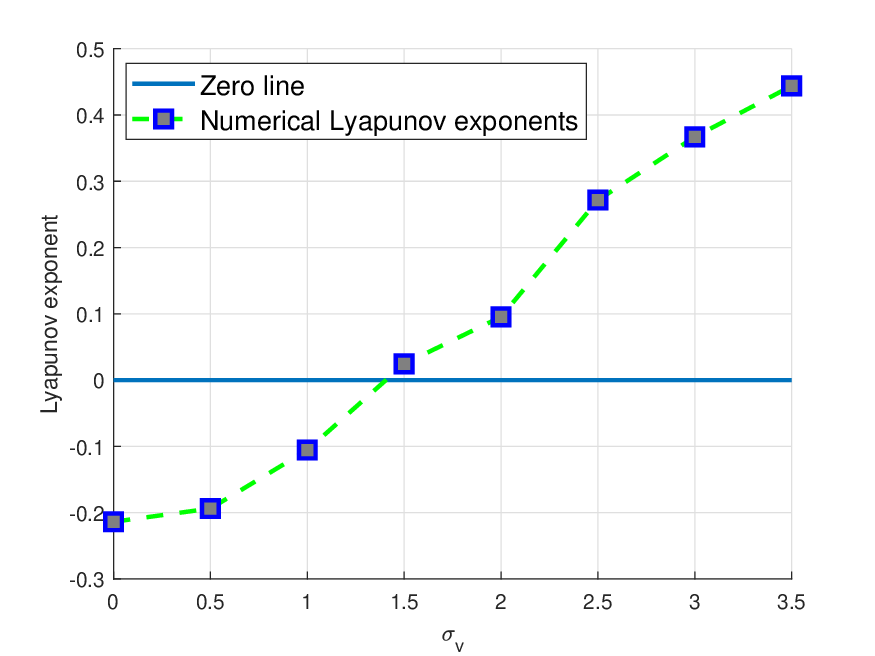}
					\caption{The computed Lyapunov exponents.}
					\label{fig:LR_destabilize_LE}
				\end{subfigure}
				\caption{Destabilization by noise with different noise intensities.}
				\label{fig:LR_destabilize}
			\end{figure}
			
			To analyze it further, we plot the profiles of solution at final time $T = 60$ Figure \ref{fig:LR_destabilize_uv}, where it clearly shows that the solution becomes spatially inhomogeneous. Moreover, in Figure \ref{fig:LR_destabilize_LE_Fourier}, we compute the Lyapunov exponents for different eigenmodes. The green line with red markers present the Lypanov exponents of the linearized system without noise, which shows that all these exponents are negative, causing the exponential stability of the system. In the presence of noise $\sigma_v = 5$ (while $\sigma_u = 0$), the Lyapunov exponents of different modes, presented by the red line with blue markers, show that some first modes (from $k=1$ to $k=16$) become positive, leading to the instability of the stochastic system. Unlike the case of deterministic instability in Figure \ref{fig:LR_Turing_LE_Fourier}, the unstable modes seem to \textit{always contain the first eigenmodes}.
			\begin{figure}[H]
				\begin{subfigure}[t]{0.45\textwidth}
					\includegraphics[width=6cm]{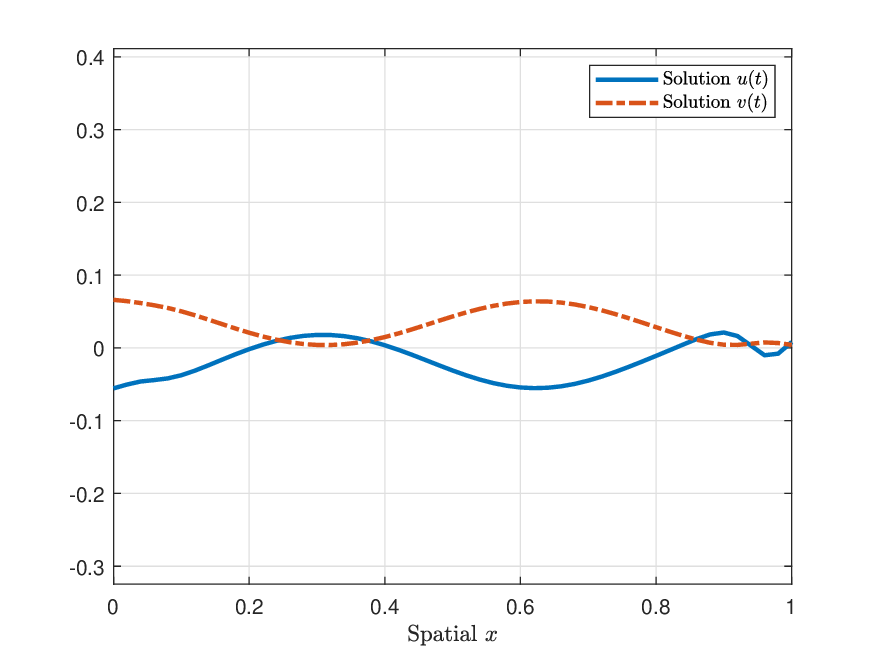}
					\caption{The profiles of $(u,v)$ at the final time $T = 60$.}
					\label{fig:LR_destabilize_uv}
				\end{subfigure} \hspace*{.1in}
				\begin{subfigure}[t]{0.45\textwidth}
					\includegraphics[width=6cm]{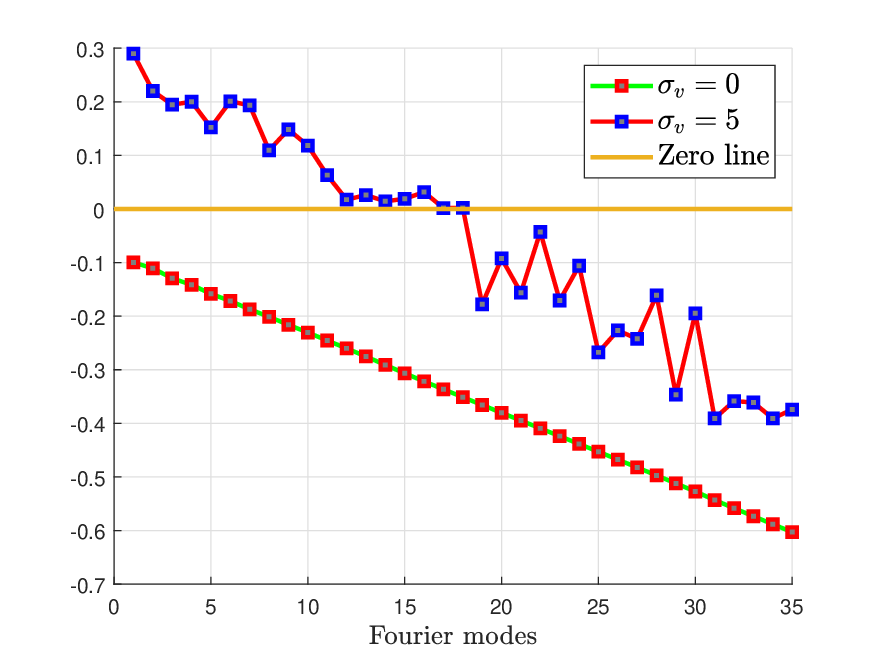}
					\caption{The Lyapunov exponents of eigenmodes from $k=1$ to $k = 35$.}
					\label{fig:LR_destabilize_LE_Fourier}
				\end{subfigure}
				\caption{The final profiles of solutions and Lyapunov exponents for different eigenmodes.}
				\label{fig:LR_destabilize_1}
			\end{figure}
			
			\section{Numerical Simulations for Nonlinear Systems}
			In this section, we investigate the effect of noise to the stability of nonlinear system \eqref{B-system} through numerical simulations. The theoretical investigation is beyond the scope of this paper and is therefore left for future studies.
			
			\subsection{Turing instability prevented by noise}
			We consider $\Omega = (0,1)$ and the following parameters
			\begin{equation*} 
				d_u = 5\times 10^{-5}, \quad d_v = 2\times 10^{-3}, \quad A = 1, \quad B = 1.8,
			\end{equation*}
			the time horizon $T = 50$, and the discretization $\delta x = 0.02$ and $\delta t = 0.005$. The initial data is set to be
			\begin{equation*} 
				u_0(x) = u_\infty + 10^{-1}\xi_1 (1+ \cos(3x)), \quad v_0(x) = v_\infty + 10^{-1}\xi_2 (1+\sin(3x))
			\end{equation*}
			where $\xi_1$ and $\xi_2$ are random numbers between $0$ and $1$.
			
			\medskip
			\noindent\textbf{The deterministic case $\sigma_u = \sigma_v = 0$}. In this case, the Turing condition is satisfied and therefore, the deterministic system possesses Turing instability, see Figure \ref{fig:unstable}. In this case, we can see from Figures \ref{fig:unstable_u} and \ref{fig:unstable_v} that the solution converges to a stable spatially inhomogeneous steady state, which is depicted in Figure \ref{fig:unstable_uv_final}.
			\begin{figure}[H]
				\centering
				\begin{subfigure}{0.31\textwidth}
					\includegraphics[width=5.2cm]{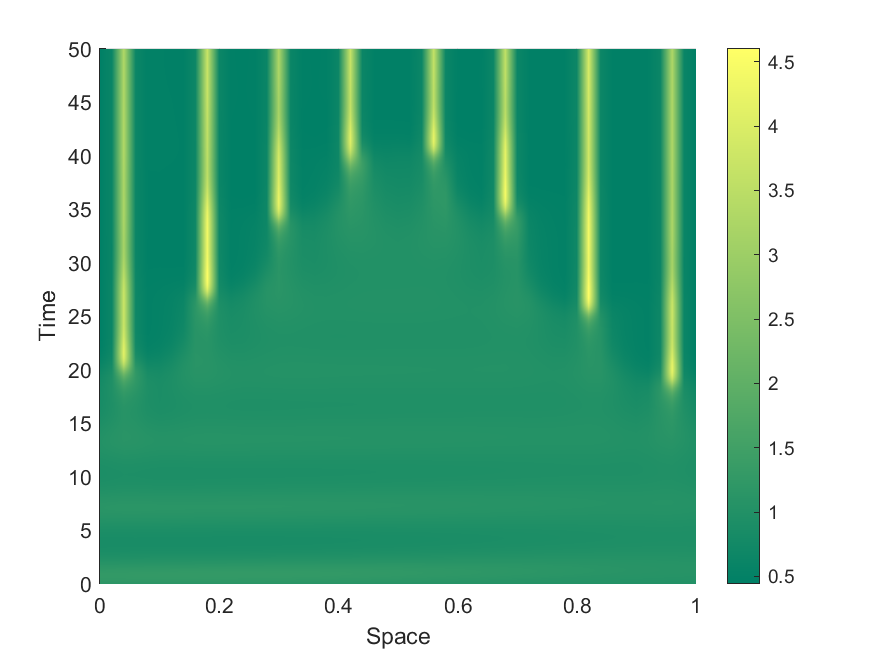}
					\caption{Evolution of $u$}
					\label{fig:unstable_u}
				\end{subfigure} \hspace{.5in}
				\begin{subfigure}{0.31\textwidth}
					\includegraphics[width=5.2cm]{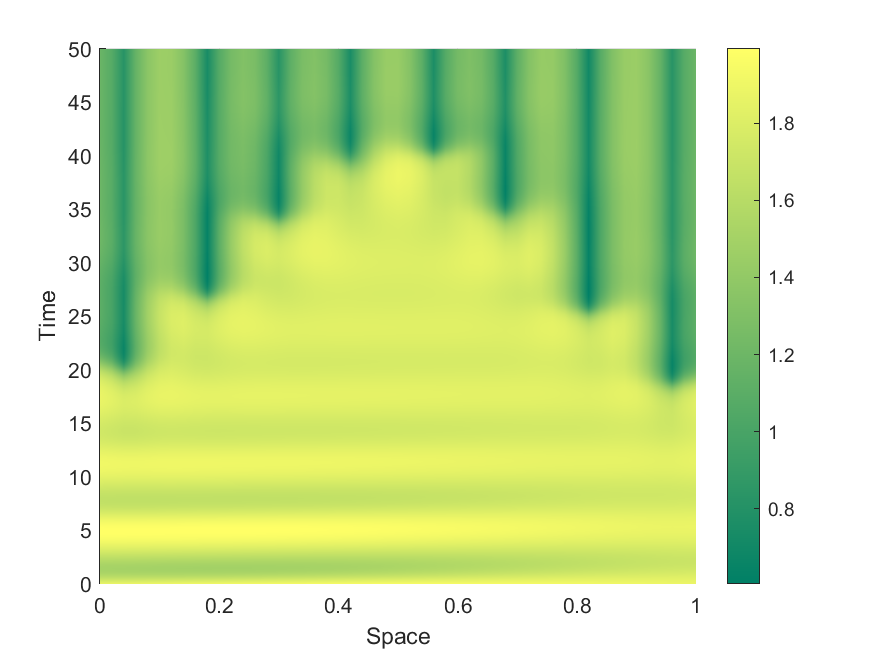}
					\caption{Evolution of $v$}
					\label{fig:unstable_v}
				\end{subfigure}\\
				\begin{subfigure}[t]{0.31\textwidth}
					\includegraphics[width=5.2cm]{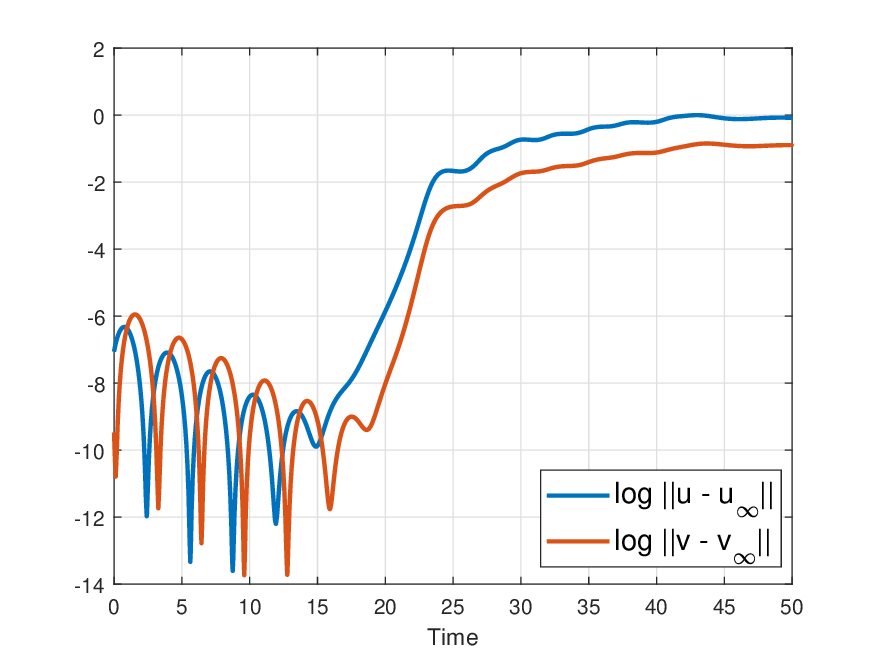}
					\caption{Evolution of the norm}
					\label{fig:unstable_norm}
				\end{subfigure} \hspace{.5in}
				\begin{subfigure}[t]{0.31\textwidth}
					\includegraphics[width=5.2cm]{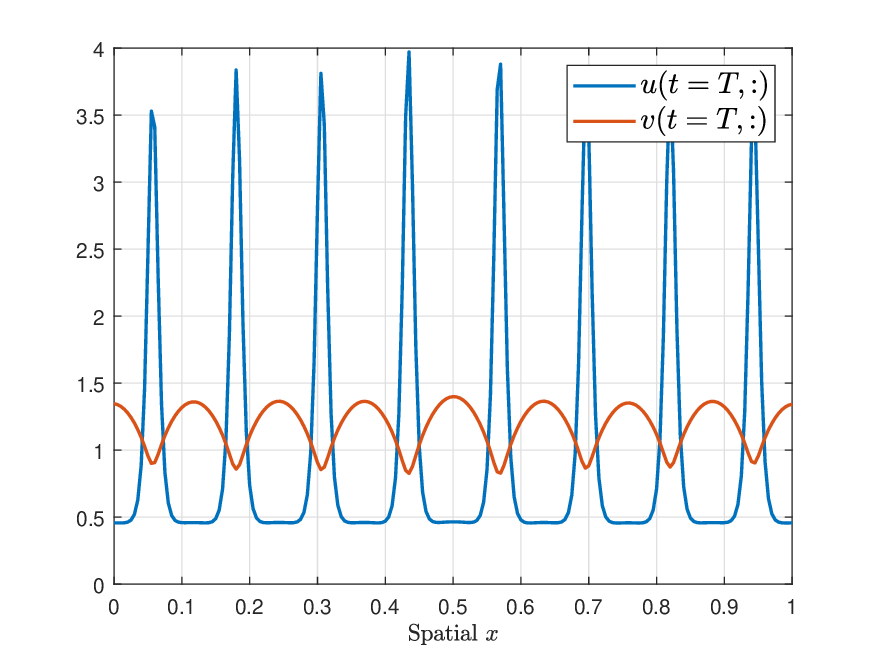}
					\caption{Profiles of $u$ and $v$ at final time $T = 50$.}
					\label{fig:unstable_uv_final}
				\end{subfigure}
				\caption{Turing instability without noise.}
				\label{fig:unstable}
			\end{figure}
			
			\medskip
			\noindent \textbf{The case of the same noise intensities $\sigma_u = \sigma_v = 1.25$}. In this case, the suppression of Turing instability is proved theoretically for the linearized system in the previous section. Here, numerical simulation suggests that it is also true for the nonlinear regime. Figure \ref{fig:noise_stabilize_norm} shows that the solution decays exponentially to the equilibrium $(u_\infty, v_\infty)$. The evolution of $u$ and $v$ are shown in Figures in \ref{fig:noise_stabilize_u} and \ref{fig:noise_stabilize_v}.
			\begin{figure}[H]
				\centering
				\begin{subfigure}{0.31\textwidth}
					\includegraphics[width=5.2cm]{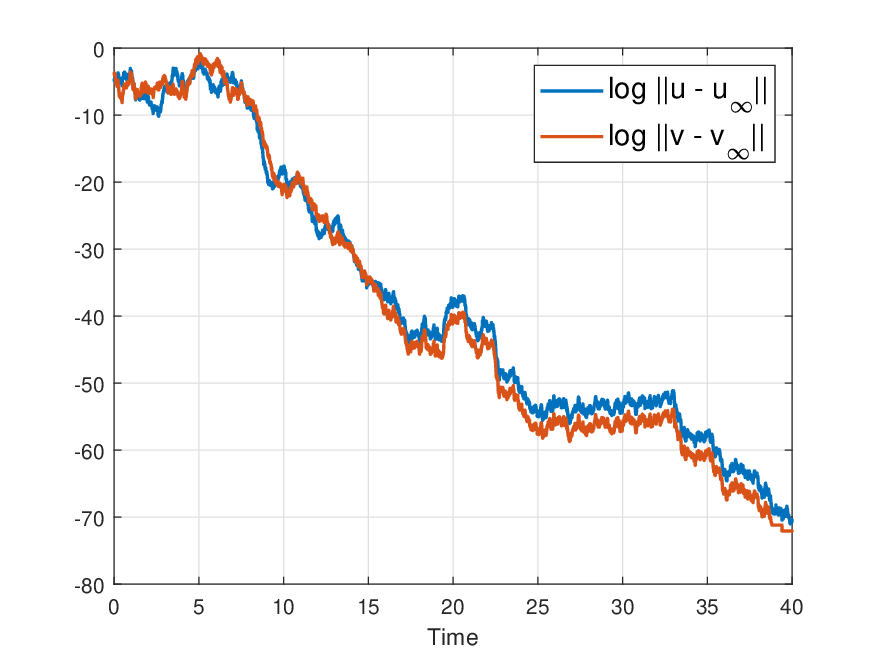}
					\caption{Evolution of the norm}
					\label{fig:noise_stabilize_norm}
				\end{subfigure}\hspace{.1in}
				\begin{subfigure}{0.31\textwidth}
					\includegraphics[width=5.2cm]{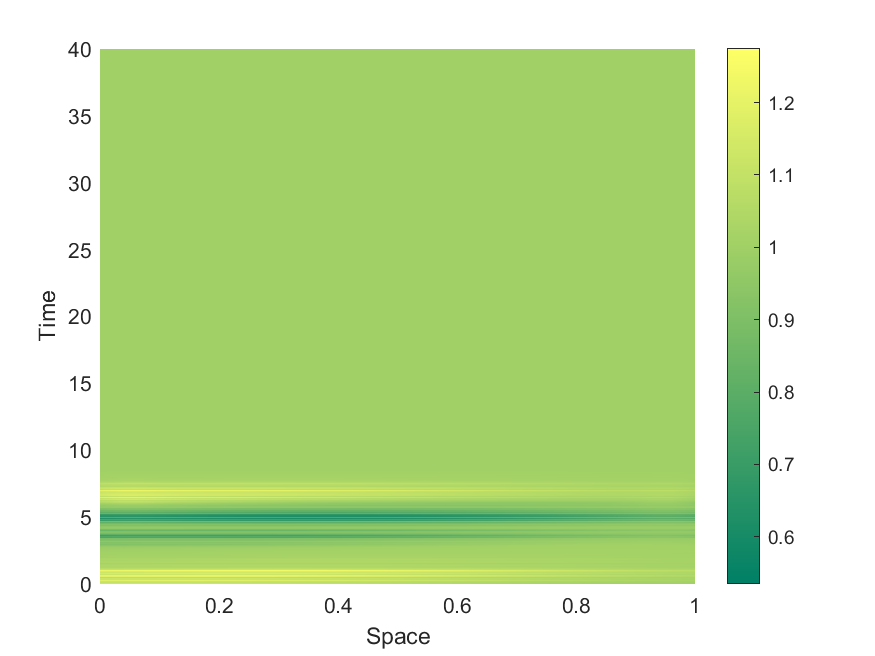}
					\caption{Evolution of $u$}
					\label{fig:noise_stabilize_u}
				\end{subfigure}\hspace{.1in}
				\begin{subfigure}{0.31\textwidth}
					\includegraphics[width=5.2cm]{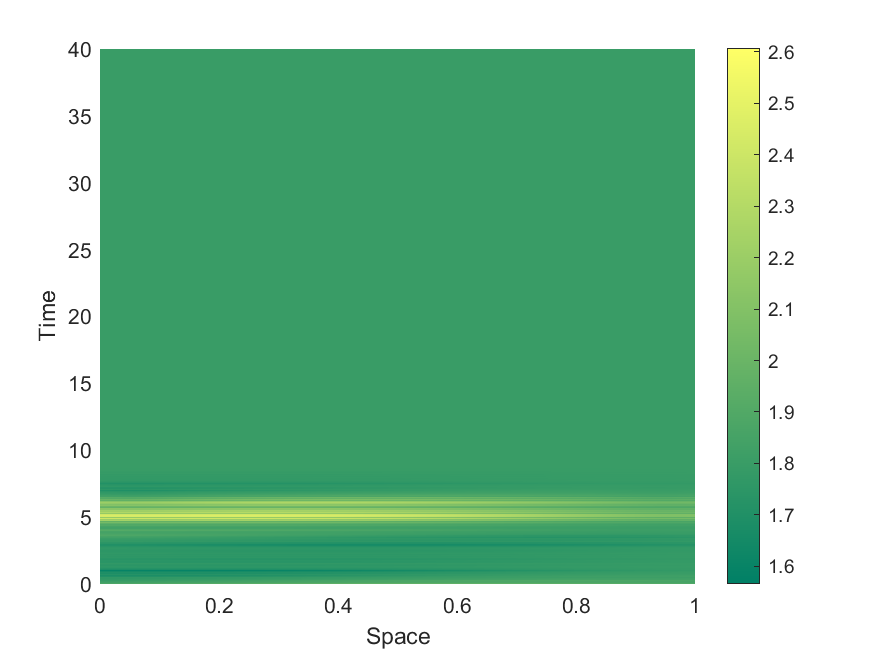}
					\caption{Evolution of $v$}
					\label{fig:noise_stabilize_v}
				\end{subfigure}
				\caption{Stabilization by noise with the same noise intensities.}
			\end{figure}
			
			\subsection{Turing instability induced by noise}
			We consider the following parameters
			\begin{equation*} 
				d_u = 0.001, \quad d_v = 0.002, \quad A = 1, \quad B = 1.95,
			\end{equation*}
			the time horizon $T = 60$, and the descretization $\delta x = 0.02$ and $\delta t = 0.005$.
			
			\medskip
			\noindent \textbf{The deterministic case $\sigma_u = \sigma_v = 0$}. In this case, the Turing condition is not satisfied as $d_u/d_v$ is not sufficiently small (see e.g. \cite[Chapter 7]{perthame2015parabolic}, and therefore we have the convergence towards the equilibrium $u_\infty = A = 2$ and $v_\infty = B/A = 2$, see Figure \ref{fig:deterministic}.
			\begin{figure}[H]
				\centering
				\begin{subfigure}{0.31\textwidth}
					\includegraphics[width=5.2cm]{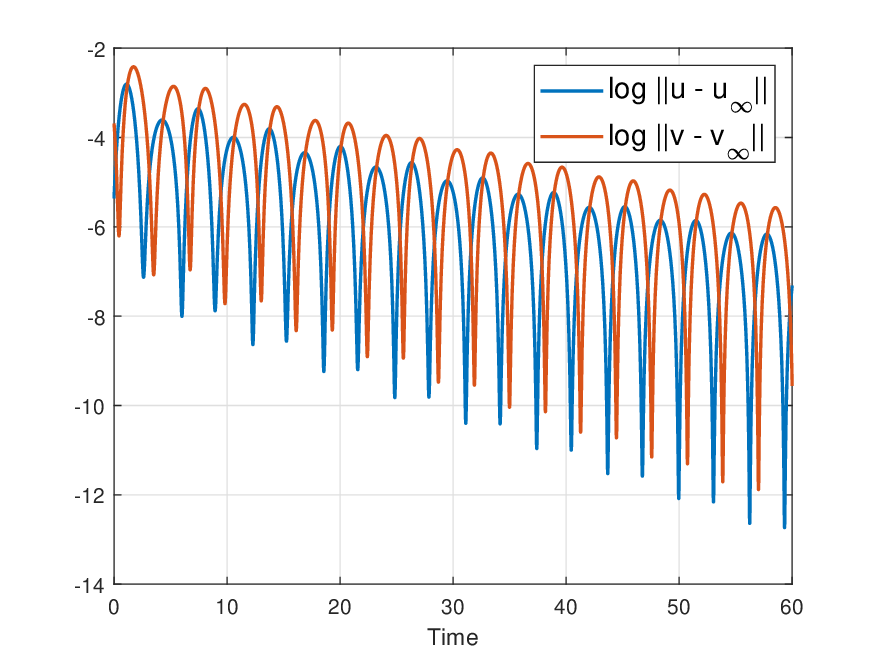}
					\caption{Evolution of the norm}
					\label{fig:deterministic_norm}
				\end{subfigure}\hspace{.1in}
				\begin{subfigure}{0.31\textwidth}
					\includegraphics[width=5.2cm]{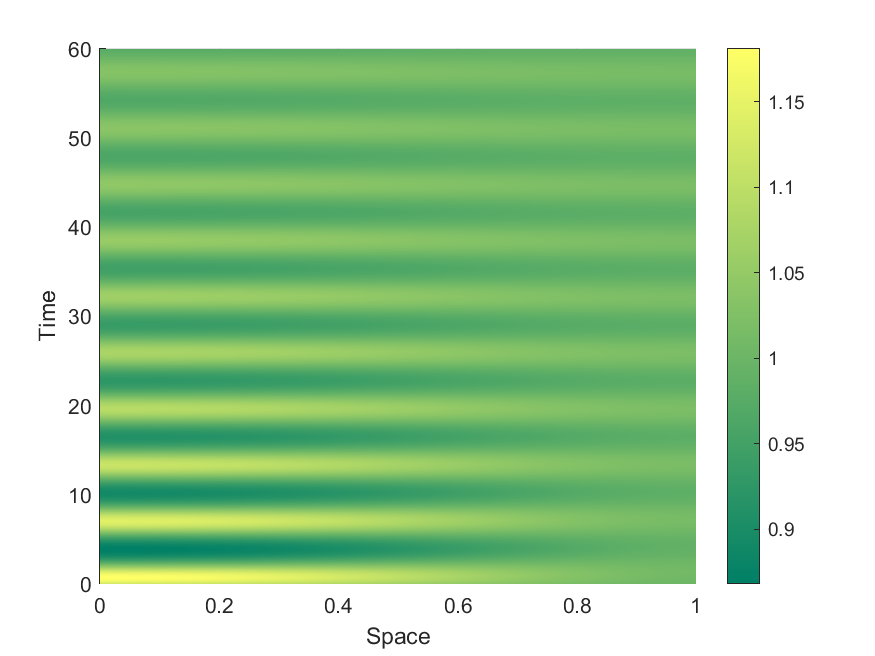}
					\caption{Evolution of $u$}
					\label{fig:deterministic_u}
				\end{subfigure}\hspace{.1in}
				\begin{subfigure}{0.31\textwidth}
					\includegraphics[width=5.2cm]{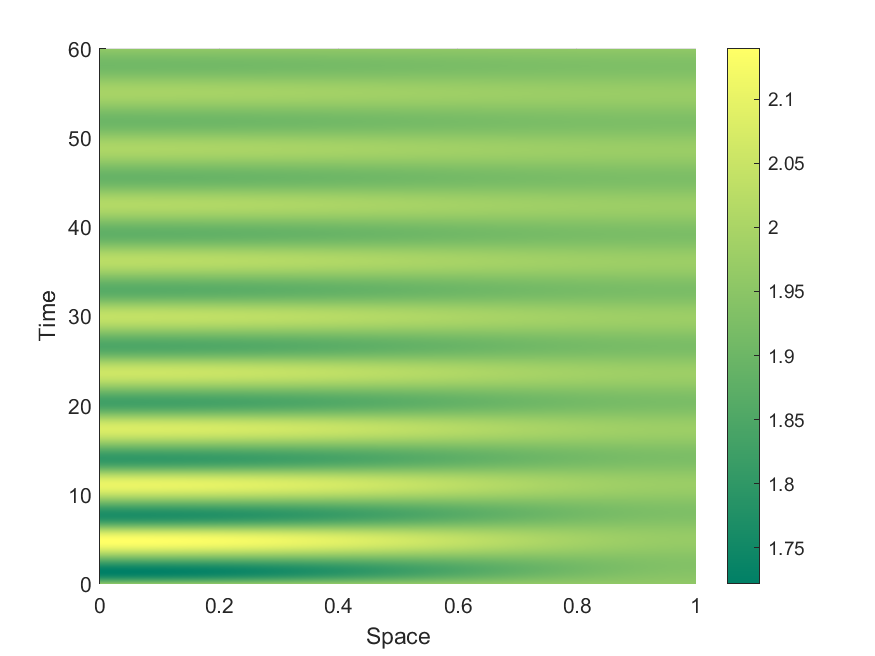}
					\caption{Evolution of $v$}
					\label{fig:deterministic_v}
				\end{subfigure}
				\caption{Deterministic stability.}
				\label{fig:deterministic}
			\end{figure}
			
			\medskip
			\noindent \textbf{The case with small noise intensities $\sigma_u = 0$ and $\sigma_v = 0.2$}. For small noise, the stability of the system is still in effect, see Figure \ref{fig:small_noise}. Similar to the linearized case, for small magnitude of $\sigma_v$, the disturbance of noise is not sufficient to drive the system away from the stable regime. The solution still converges exponentially to equilibrium, with some disruption, as shown in Figure \ref{fig:small_noise}. This can be compared with Figure \ref{fig:LR_destabilize_LE}, where with small magnitude of $\sigma_v$ the Lyapunov exponent is still negative, meaning that linearized system, and consequently the nonlinear system, is still stable.
			\begin{figure}[H]
				\centering
				\begin{subfigure}{0.31\textwidth}
					\includegraphics[width=5.2cm]{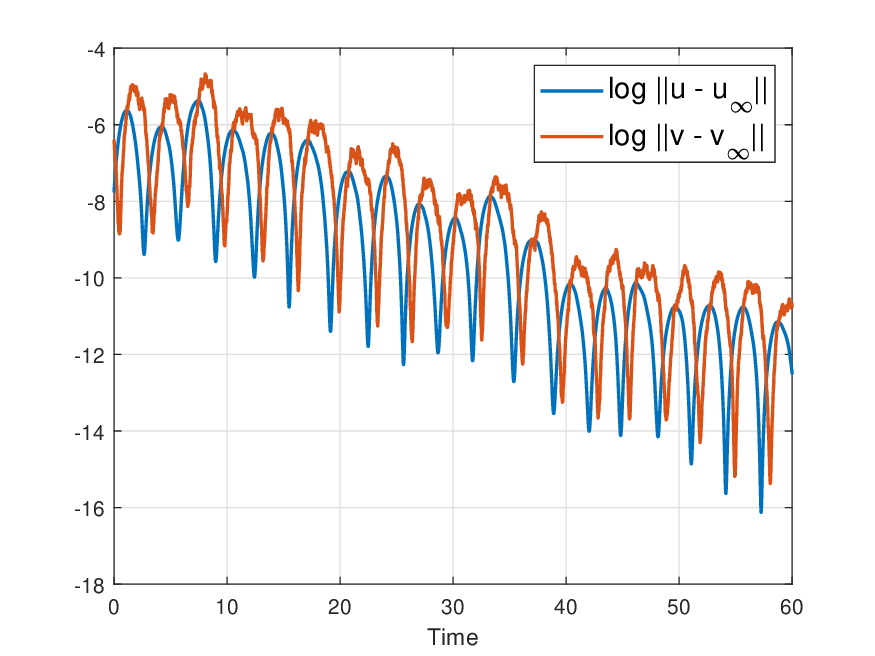}
					\caption{Evolution of the norm}
					\label{fig:small_noise_norm}
				\end{subfigure}\hspace{.1in}
				\begin{subfigure}{0.31\textwidth}
					\includegraphics[width=5.2cm]{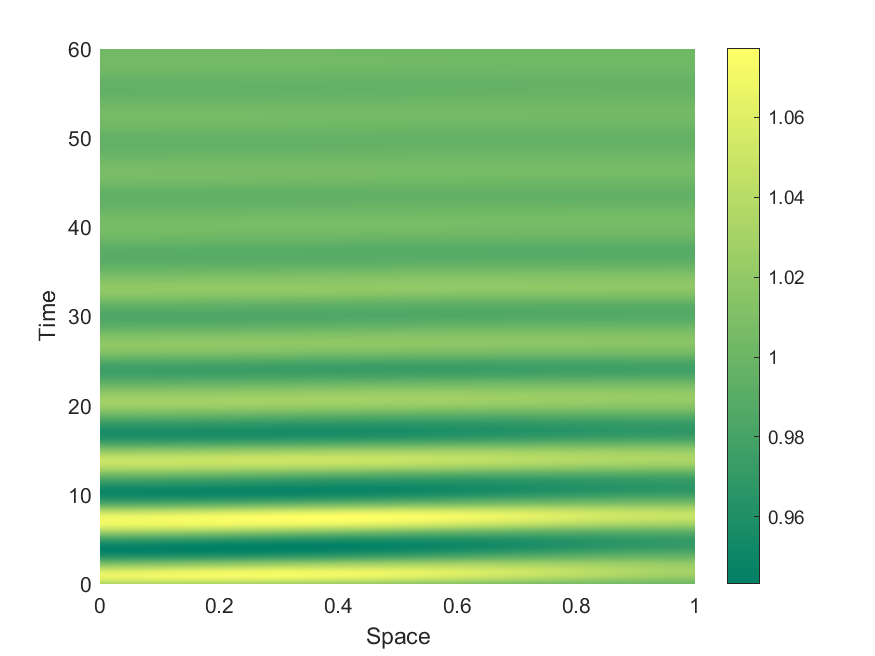}
					\caption{Evolution of $u$}
					\label{fig:small_noise_u}
				\end{subfigure}\hspace{.1in}
				\begin{subfigure}{0.31\textwidth}
					\includegraphics[width=5.2cm]{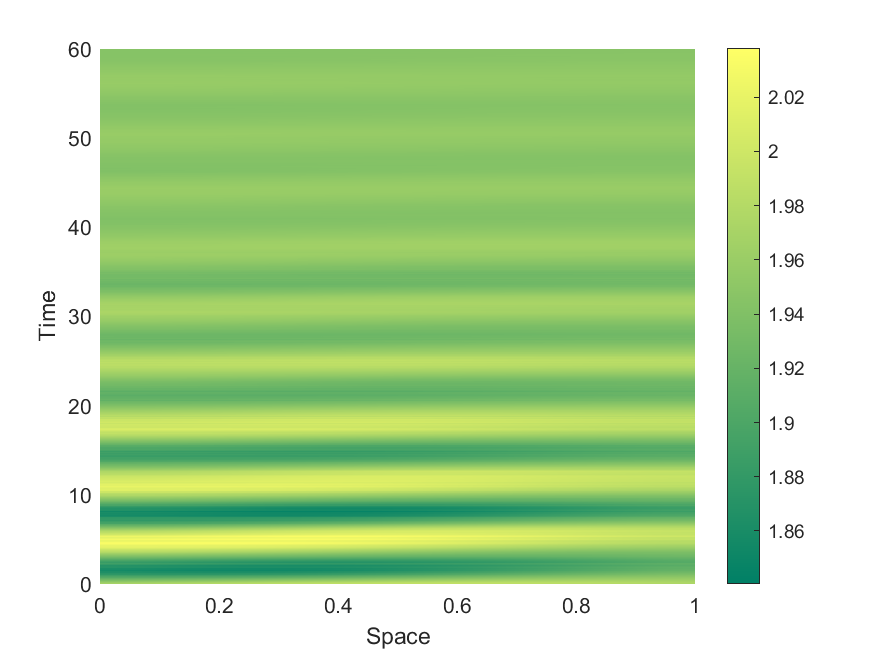}
					\caption{Evolution of $v$}
					\label{fig:small_noise_v}
				\end{subfigure}
				\caption{Small noise still preserves the stability.}
				\label{fig:small_noise}
			\end{figure}
			
			\medskip
			\noindent \textbf{The case with bigger noise intensities $\sigma_u = 0$ and $\sigma_v = 3.0$}. Finally, when the noise intensity $\sigma_v$ is sufficiently large, the noise destabilizes the system as shown in Figure \ref{fig:big_noise}. The solution is no longer converging to the spatially homogeneous steady state $(u_\infty, v_\infty)$. However, unlike the case of deterministic instability, due to the random noise, the solution also does not converge to a stable spatially inhomogeneous steady state, but fluctuates with different profiles. The large time behavior of solution in this case remains open, even in the numerical simulation.
			\begin{figure}[h!]
				\centering
				\begin{subfigure}{0.31\textwidth}
					\includegraphics[width=5.2cm]{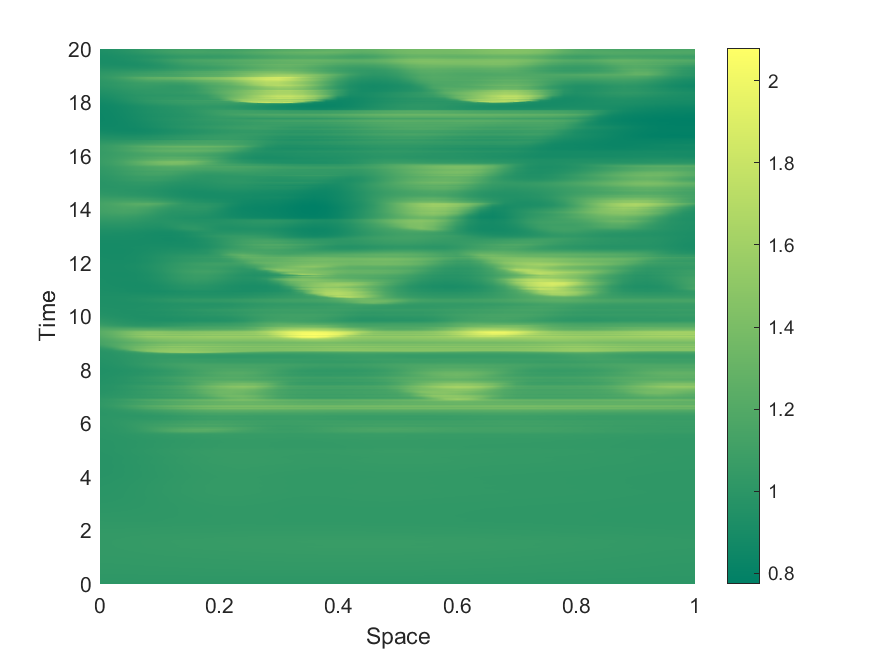}
					\caption{Evolution of $u$}
					\label{fig:big_noise_u}
				\end{subfigure} \hspace{.5in}
				\begin{subfigure}{0.31\textwidth}
					\includegraphics[width=5.2cm]{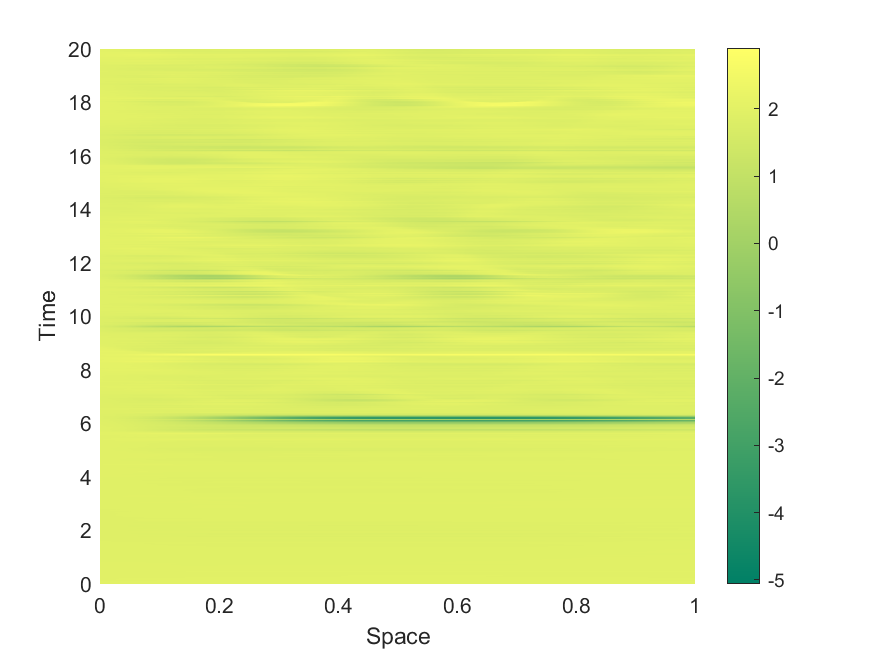}
					\caption{Evolution of $v$}
					\label{fig:big_noise_v}
				\end{subfigure}\\
				\begin{subfigure}[t]{0.31\textwidth}
					\includegraphics[width=5.2cm]{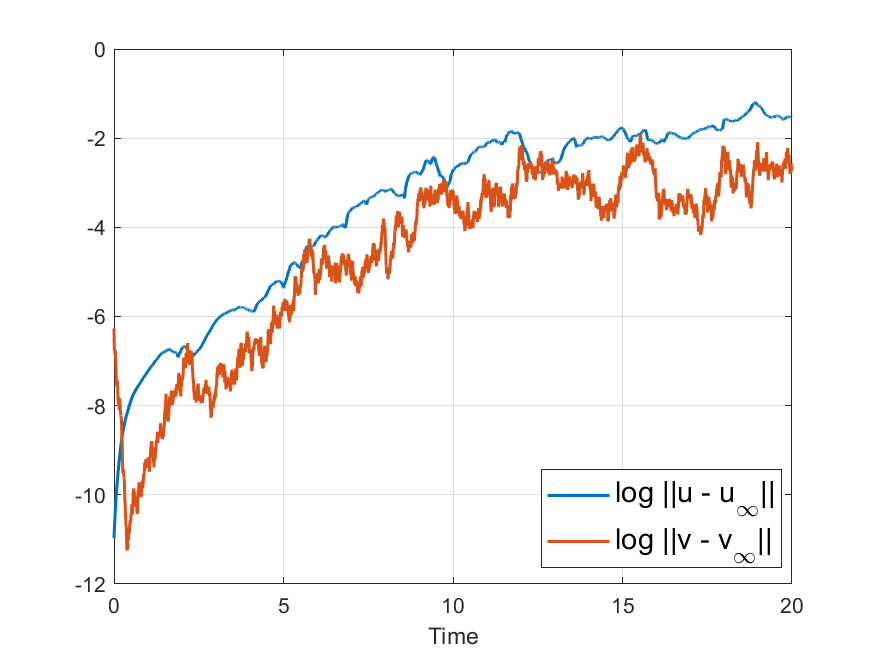}
					\caption{Evolution of the norm}
					\label{fig:big_noise_norm}
				\end{subfigure} \hspace{.5in}
				\begin{subfigure}[t]{0.31\textwidth}
					\includegraphics[width=5.2cm]{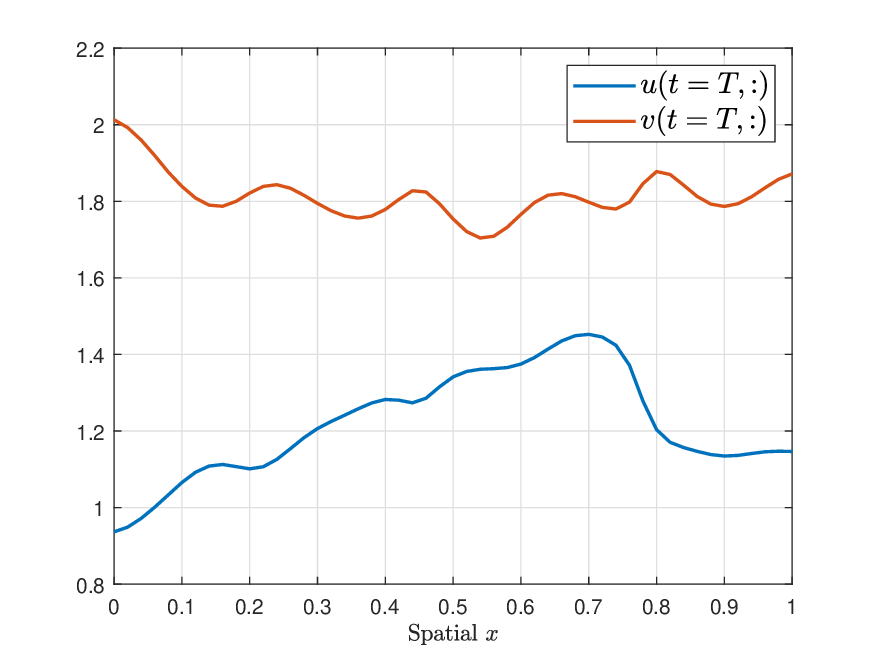}
					\caption{Profiles of $u$ and $v$ at final time $T = 20$.}
					\label{fig:big_noise_uv_final}
				\end{subfigure}
				\caption{Larger noise intensities lead to instability.}
				\label{fig:big_noise}
			\end{figure}
			%
			
			\section{Conclusion and Future Work}\label{conclusion}
			
			In this work, we have demonstrated that a multiplicative noise in Brusselator can have suppression or induced effect on Turing instability. More precisely, when both the noise intensities on both species are the same, then the Turing instability is suppressed if the intensities are sufficiently large. This is also theoretically proved by analyzing the linearized system around equilibrium, and using explicit expression of solutions to SDE systems when the deterministic and stochastic parts are commutative. On the other hand, when one noise intensity is zero, multiplicative noise tends to destabilize stable system, which is shown and analyzed by numerical simulations. These effects are shown through numerical examples to hold also for nonlinear Brusselator system. This finding represents an interesting perspective of the complex interplay between Turing stability and stochastic perturbations. Furthermore, we believe that the insights gained from this study can be extended to a wider class of nonlinear dynamical systems.
			
			\medskip
			There are certainly several important open questions to be investigated in future works, including
			\begin{itemize}
				\item can the suppression by noise with the same noise intensities be proved theoretically as for the case of scalar equation in e.g. \cite{caraballo2006stabilization,caraballo2007effect,kwiecinska1999stabilization}?
				\item how does one approach to prove the destabilizing effect when $\sigma_u= 0$ and $\sigma_v\ne 0$ since explicit expressions are not possible, and the general theory as in \cite{mao2007stochastic} seem not applicable?
			\end{itemize}
			

			\subsection*{Acknowledgement}
			Qasim Khan acknowledges the funding of the Ernst Mach (worldwide) scholarship from OeAD, Austria's Agency for Education and Internationalization, Mobility Programmers and Cooperation. The work was completed during the first author's visit to University of Graz, and the university's hospitality is greatly acknowledged. Bao Q. Tang is supported by the FWF project ``Quasi-steady-state approximation for PDE'', number I-5213. A. Suen is
			partially supported by Hong Kong General Research Fund (GRF) grant project number
			18300821, 18300622 and 18300424, and the EdUHK Research Incentive Award project
			titled ``Analytic and numerical aspects of partial differential equations''.
			
			\subsection*{CRediT authorship contribution statement}
			All authors contribute equally in this work. 
			
			\subsection*{Declaration of competing interest}
			The authors declare that they have no known competing financial interests or personal relationships that could have appeared
			to influence the work reported in this paper.
			
			\subsection*{Data availability}
			No data was used for the research described in the article.
			
			\newcommand{\etalchar}[1]{$^{#1}$}

			\end{document}